\documentclass[12pt,leqno,a4paper,final]{amsart}
\usepackage{amsmath, amsthm, verbatim, amsfonts, amssymb, color}
\usepackage{mathrsfs}
\usepackage{dsfont}
\usepackage{stix2}
\usepackage[a4paper,margin=1in]{geometry}

\theoremstyle{plain}
\newtheorem{theorem}{Theorem}
\newtheorem{corollary}[theorem]{Corollary}
\newtheorem{lemma}[theorem]{Lemma}

\newtheorem*{theorem*}{Theorem}
\newtheorem*{corollary*}{Corollary}
\newtheorem*{lemma*}{Lemma}
\newtheorem*{proposition*}{Proposition}

\theoremstyle{definition}
\newtheorem{remark}[theorem]{Remark}
\newtheorem*{remark*}{Remark}

\renewcommand\labelenumi{\textup{\alph{enumi})}}
\renewcommand\theenumi\labelenumi

\renewcommand{\Im}{\ensuremath{\operatorname{Im}}}
\newcommand{\eup}{\mathrm{e}}
\newcommand{\iup}{\mathrm{i}}

\makeatletter
\def\@makefnmark{\hbox{(\@textsuperscript{\normalfont\@thefnmark})}}
\makeatother

\newcommand{\primo}{1\textsuperscript{o}\;}
\newcommand{\secundo}{2\textsuperscript{o}\;}
\newcommand{\tertio}{3\textsuperscript{o}\;}
\newcommand{\quarto}{4\textsuperscript{o}\;}
\newcommand{\quinto}{5\textsuperscript{o}\;}
\newcommand{\sexto}{6\textsuperscript{o}\;}
\newcommand{\septimo}{7\textsuperscript{o}\;}
\newcommand{\octavo}{8\textsuperscript{o}\;}
\newcommand{\nono}{9\textsuperscript{o}\;}

\makeatletter
\@namedef{subjclassname@2020}{%
  \textup{2020} Mathematics Subject Classification}
\makeatother

\newcommand\Ee{\mathds{E}}
\newcommand\nat{\mathds{N}}

\newcommand\integer{\mathds{Z}}
\newcommand\real{{\mathds{R}}}
\newcommand\comp{{\mathds{C}}}

\newcommand\I{\mathds{1}}
\newcommand\Pp{\mathds{P}}

\newcommand\supp{\mathop{\mathrm{supp}}}

\newcommand\rd{{\mathds{R}^d}}

\newcommand{\Fscr}{\mathscr{F}}
\newcommand{\Fcal}{\mathcal{F}}

\newcommand{\Dcal}{\mathcal{D}}

\newcommand{\Scal}{\mathcal{S}}
\newcommand{\Tcal}{\mathcal{T}}

\begin{document}
\begin{flushright}\small
\underline{To appear in} \emph{Probability and Mathematical Statistics}\\
\bigskip\bigskip
\end{flushright}
\title[L\'evy Processes, Generalized Moments and Uniform Integrability]{\bfseries L\'evy Processes, Generalized Moments and Uniform Integrability}

\author[D.~Berger]{David Berger}
\author[F.~K\"{u}hn]{Franziska K\"{u}hn}
\author[R.L.~Schilling]{Ren\'e L.\ Schilling}
\address{TU Dresden\\ Fakult\"{a}t Mathematik\\ Institut f\"{u}r Mathematische Stochastik\\ 01062 Dresden, Germany}
\email{david.berger2@tu-dresden.de}
\email{franziska.kuehn1@tu-dresden.de}
\email{rene.schilling@tu-dresden.de}

\thanks{\emph{Acknowledgement}. The comments of an anonymous referee helped to improve the presentation of this paper. Financial support through the DFG-NCN Beethoven Classic 3 project SCHI419/11-1 \& NCN 2018/31/G/ST1/02252 is gratefully  acknowledged.}

\subjclass[2020]{60G51; 60G44; 60G40; 26A12; 26B35.}
\keywords{L\'evy process; Dynkin's formula; generalized moment; Gronwall's inequality; local martingale; condition D; condition DL}

\maketitle

\bigskip
\begin{quote}\begin{small}\noindent
    \textsc{Abstract.} We give new proofs of certain equivalent conditions for the existence of generalized moments of a L\'evy process $(X_t)_{t\geq 0}$; in particular, the existence of a generalized $g$-moment is equivalent to the uniform integrability of $(g(X_t))_{t\in [0,1]}$. As a consequence, certain functions of a L\'evy process which are integrable and local martingales are  already  true martingales. Our methods extend to moments of stochastically continuous additive processes, and we give new, short proofs for the characterization of lattice distributions and the transience of L\'evy processes.
\end{small}
\end{quote}

\section{Introduction}
\bigskip\noindent
A \emph{generalized moment} of a stochastic process $(X_t)_{t\geq 0}$ is an expression of the form $\Ee\left[g(X_t)\right]$. Such moments arise naturally when studying Markov semigroups. It is a classical result that for a L\'evy process and a submultiplicative function $g$ the $g$-moment exists if, any only if, $g(x)\I_{\{|x|>1\}}$ is integrable w.r.t.\ the jump measure of the process (Section~\ref{sec-gen}). Throughout the analysis and probability literature, further equivalent criteria for the existence of $g$-moments can be found. In this note we collect these criteria, add a few more equivalences, and give a unified presentation with novel proofs. These proofs frequently exploit martingale techniques and the Markov property, rather than the translation invariance of a L\'evy process, which means that many implications -- alas, not all -- remain valid in more general situations, cf.\  Remark~\ref{gen-077} on page~\pageref{gen-077}. A summary on the existing literature is also given in Remark~\ref{gen-077}. Our arguments are based on Dynkin's formula (which is a consequence of the martingale nature of the process) and Gronwall's lemma, and this technique can also be used (Section~\ref{sec-doo}) to show that certain functions of a L\'evy process $(f(X_t))_{t\geq 0}$ which are both a local martingale and integrable, i.e.\ $\Ee |f(X_t)|<\infty$, are already proper martingales. With some minor changes our methods extend to stochastically continuous additive processes (Section~\ref{sec-add}), and the criteria in Theorem~\ref{add-10} seem to be new, extending earlier work by Fujiwara~\cite{fujiwara} and Klass \& Yang~\cite{klass}. In the last part of the paper (Section~\ref{sec-lat}) we further apply our results to get a very short proof of the characterization of infinitely divisible lattice distributions and a martingale approach to the transience of L\'evy processes.

Let us recall a few key concepts and techniques which will be needed later on. Most of our notation is standard or self-explanatory; we use $|x|_{\ell^p}^p := \sum_{k=1}^d |x_k|^p$ with the usual modification if $p=\infty$. We write $\|f\|_{L^1(\rd,g)}:=\int_{\rd} |f(x)|\, g(x) dx$ for the weighted $L^1$-norm  and  $L^1(\rd,g)$ for the corresponding $L^1$-space (with a nonnegative, measurable weight function $g:\rd\to [0,\infty)$).

\subsection{L\'evy processes}
A L\'evy process $X=(X_t)_{t\geq 0}$ is a stochastic process with values in $\rd$, stationary and independent increments and right-continuous sample paths with finite left-hand limits (c\`adl\`ag). Our standard references for L\'evy processes are Sato \cite{sato} (for probabilistic properties) and Jacob \cite{jac-1} and \cite{barca} (for analytic aspects). It is well known that a stochastic process $X$ is a L\'evy process if it has c\`adl\`ag paths and if its conditional characteristic function is of the form
\begin{gather*}
    \Ee \left[\eup^{\iup \xi\cdot (X_t-X_s)} \mid \Fscr_s\right]
    = \eup^{-(t-s)\psi(\xi)},\quad 0\leq s\leq t,\; \xi\in\rd,
\end{gather*}
where $\Fscr_s = \sigma(X_r, r\leq s)$ is the natural filtration of $X$. The \emph{characteristic exponent} $\psi:\rd\to\comp$ is uniquely determined by the \emph{L\'evy--Khintchine formula}
\begin{gather}\label{intro-e01}
    \psi(\xi)
    = -\iup b\cdot\xi + \frac 12 Q\xi\cdot\xi + \int_{\rd\setminus\{0\}}\left(1-\eup^{\iup  \xi\cdot x}+\iup\xi\cdot x\I_{(0,1)}(|x|)\right)\nu(dx).
\end{gather}
The \emph{L\'evy triplet} $(b,Q,\nu)$ where $b\in\rd$, $Q\in\real^{d\times d}$ (a positive semidefinite matrix) and $\nu$ (a Radon measure on $\rd\setminus\{0\}$ such that $\int_{\rd\setminus\{0\}} \min\{|x|^2,1\}\,\nu(dx)<\infty$) uniquely describe $\psi$.

Using the characteristic exponent we can determine the infinitesimal generator $A$ of the process $X$ either as \emph{pseudo-differential operator}
\begin{align}
\notag
    A u(x) &= -\psi(D)u(x) = \Fcal^{-1}[-\psi \Fcal u](x), \quad u \in \Scal(\rd),
\intertext{where $\Fcal u(\xi) = (2\pi)^{-d}\int_\rd \eup^{-\iup\xi\cdot x} u(x)\,dx$ is the Fourier transform and $\Scal(\rd)$ is the Schwartz space of rapidly decreasing smooth functions, or as integro-differential operator}
\label{intro-e02}
    A u(x) &= Lu(x) + Ju(x) + Ku(x)
\intertext{where $L$, $J$ and $K$ are again linear operators: $L$ is the local part, $J$ takes into account the small jumps and $K$ the large jumps, i.e.}
\notag
    L u(x) &= b\cdot\nabla u(x) + \frac 12\nabla \cdot Q\nabla u(x), \\
\notag
    J u(x) &= \int_{0<|y|<1} \left(u(x+y)-u(x)-y\cdot\nabla u(x)\right)\nu(dy),\\
\notag
    K u(x) &= \int_{|y|\geq 1} \left(u(x+y)-u(x)\right)\nu(dy).
\end{align}
The precise form of the domain $\Dcal(A)$ of $A$  (as closed operator on the Banach space of continuous functions vanishing at infinity $(C_\infty(\rd),\|\cdot\|_\infty)$) is not known; but both the test functions $C_c^\infty(\rd)$ and the Schwartz spaces $\Scal(\rd)$ are operator cores. On the other hand, the expression \eqref{intro-e02} has a pointwise meaning for every positive $u\in C^2(\rd)$: using a Taylor expansion, it is easy to see that for every $x\in \real^d$ one has $|(L+J)u(x)|\leq C\sup_{|x-y|\leq 1} (|\nabla u(y)|+|\nabla^2 u(y)|)<\infty$ for some finite constant $C>0$ while positivity is needed for $Ku(x)\in (-\infty,\infty]$.\footnote{The latter may be replaced by a polynomial growth bound on $u$, $|u(x)|\leq C(1+|x|^p)$ and a moment condition $\int_{|y|>1}|y|^p\nu(dy)<\infty$ on $\nu$.}  We will continue to use the notation $Au(x)$ despite the fact that $C^2(\rd)\not\subset \Dcal(A)$; in particular, we allow that $Au(x)$ takes values in $(-\infty,\infty]$.

The \emph{transition semigroup} $(P_t)_{t\geq 0}$ corresponding to the generator $A$ or the process $X$ is given by $P_t u(x) = \Ee \left[u(x+X_t)\right]$. Its adjoint, $P_t^*u(x) = \Ee \left[u(x-X_t)\right]$ is the transition semigroup of the L\'evy process $-X = (-X_t)_{t\geq 0}$.

\subsection{Additive processes}
An \emph{additive process} is a stochastic process $(X_t)_{t \geq 0}$ with independent increments, values in $\rd$ and starting from $0$. In particular, every L\'evy process is a stochastically continuous additive process. Because of the non-stationarity of the increments, an additive process may have fixed jump discontinuities and need not be a semimartingale; a simple --~and at the same time typical~-- example is the process $X_t = b_t$ where $(b_t)_{t \geq 0}$ is a deterministic c\`adl\`ag function of infinite variation.

For each $t\geq 0$, the law of a stochastically continuous additive process $(X_t)_{t \geq 0}$ is infinitely divisible, i.e.\ the law of $X_t$ is uniquely determined by a L\'evy--Khin\-tchine formula with a $t$-dependent characteristic triplet $(b_t,Q_t,\nu_t)$. Additivity implies that these triplets satisfy $\langle \xi, Q_s \xi\rangle \leq \langle \xi, Q_t \xi\rangle$ and $\nu_s(B)\leq \nu_t(B)$ for all $s\leq t$ and every $\xi\in \real^d$ and Borel set $B\subseteq\rd\setminus\{0\}$. Moreover, for each $\xi\in\rd$ the map
\begin{gather}\label{intro-e05}
    t\mapsto \langle \xi,b_t\rangle + \langle \xi, Q_t\xi\rangle + \int_{x\neq 0} \min\left\{1, |x|^2\right\} \nu_t(dx),
    \quad t\geq 0,
\end{gather}
is continuous. A full discussion is given in Sato~\cite[Chapter 2.9]{sato}.

\subsection{Dynkin's formula}

Let $(X_t)_{t\geq 0}$ be a L\'evy process (or a strong Markov process) and denote by $\Fscr_t = \sigma(X_s, s\leq t)$ its natural filtration and $(A,\Dcal(A))$ the infinitesimal generator. Dynkin's formula states that for every $u\in\Dcal(A)$ and every stopping time $\sigma$ with $\Ee\left[\sigma\right]<\infty$ we have
\begin{gather}\label{intro-e06}
    \Ee \left[u(X_\sigma+x)\right] - u(x) = \Ee\left[\int_{[0,\sigma)} Au(X_s)\,ds\right].
\end{gather}
There are several ways to prove this result, e.g.\ using arguments from potential theory (as in \cite[Proposition 7.31]{schilling-bm}), semigroup theory (as in \cite[Proposition VII.1.6]{revuz-yor}) or by It\^o's formula. At the heart of the argument is the fact that
\begin{gather}\label{intro-e08}
    M_t^{[u]} := u(X_t+x) - u(x) - \int_{[0,t)} Au(X_s)\,ds,\quad u\in\Dcal(A),
\end{gather}
is an $\Fscr_t$-martingale combined with a stopping argument. There are various ways to extend the class of functions $u$ for which we have some kind of Dynkin's formula.  It is clear that formula \eqref{intro-e06} can be extended to all functions $u$ with $u(X_\sigma)\in L^1(\Pp)$ and $Au(X_{s\wedge\sigma})\in L^1(ds\otimes\Pp)$. Such moment estimates will be given below.

Here we need a \emph{Dynkin inequality} which we are going to prove for positive $g\in C^2(\rd)$.
\begin{lemma}[Dynkin's inequality]\label{intro-03}
    Let $(X_t)_{t\geq 0}$ be a L\'evy process with generator $(A,\Dcal(A))$ and extend $A$ using \eqref{intro-e02} to $C^2(\rd)$. For every $g\in C^2(\rd)$ satisfying $g(x)\geq 0$ and every stopping time $\sigma$ the following inequality holds
    \begin{gather}\label{intro-e14}
        \Ee g(X_{t\wedge\sigma}) \leq g(0) + \Ee\left[\int_{[0,t\wedge\sigma)} |Ag(X_s)|\,ds\right].
    \end{gather}
\end{lemma}
\begin{proof}
Pick for every $R>0$ some cut-off function $\chi_R\in C^\infty(\rd)$ with $\I_{\overline{B}_{R+1}(0)} \leq \chi_R \leq \I_{\overline{B}_{R+2}(0)}$. Since $g\chi_R\in C_c^2(\rd)$ we know that $g\chi_R\in\Dcal(A)$, and we see that for any stopping time $\sigma$ the process $\big(M^{[g\chi_R]}_{t\wedge\sigma}\big)_{t\geq 0}$ is a martingale. Therefore we have
\begin{gather}\label{intro-e10}
    \Ee \left[(g\chi_R)(X_{t\wedge\sigma})\right] - g(0) = \Ee\left[\int_{[0,t\wedge\sigma)} A(g\chi_R)(X_s)\,ds\right].
\end{gather}
If we replace $\sigma$ by the stopping time $\sigma\wedge\tau_R$ where $\tau_R = \inf\left\{s\geq 0 \mid |X_s|\geq R\right\}$, then we can use the fact that $|X_s|\leq R$ if $s \in [0,t\wedge\sigma\wedge\tau_R)$. This implies, in particular, that $\partial^\alpha(g\chi_R)(X_s) = \partial^\alpha g(X_s)$, and we see from the integro-differential representation \eqref{intro-e02} of $A$ that for all $0\le s<t\wedge \sigma\wedge \tau_R$ the following holds:
\begin{equation}\label{intro-e12}\begin{aligned}
    A (g\chi_R)(X_s)
    &= L(g\chi_R)(X_s) + J(g\chi_R)(X_s) + K(g\chi_R)(X_s)\\
    &= b\cdot\nabla g(X_s) + \frac 12\nabla \cdot Q\nabla g(X_s) \\
    &\qquad\mbox{} + \!\!\!\int\limits_{0<|y|<1}\!\!\! \left(g(X_s+y)-g(X_s)-y\cdot\nabla g(X_s)\right)\nu(dy)\\
    &\qquad\mbox{} + \int_{|y|\geq 1} \left((g\chi_R)(X_s+y)-g(X_s)\right)\nu(dy).
\end{aligned}\end{equation}
For the second equality observe that $|X_s|\leq R$, $|X_s+y|\leq R+1$ for $|y|<1$ and that $L$ is a local operator. Since $g$ is positive, we have $g\chi_R \leq g$, and we conclude that $A (g\chi_R)(X_s) \leq Ag(X_s)$. Inserting this into \eqref{intro-e10} gives
\begin{align*}
    \Ee \left[(g\chi_R)(X_{t\wedge\sigma\wedge\tau_R})\right] - g(0)
    &\leq \Ee\left[\int_{[0,t\wedge\sigma\wedge\tau_R)} A g(X_s)\,ds\right]\\
    &\leq \Ee\left[\int_{[0,t\wedge\sigma)} |A g(X_s)|\,ds\right].
\end{align*}
Since $g\geq 0$, we can use Fatou's lemma on the left-hand side and get \eqref{intro-e14}.
\end{proof}

\subsection{Friedrichs mollifiers}
Let $j:\rd\to [0,\infty)$ be a $C^\infty$-function with compact support $\supp j\subset \overline{B}_1(0)$ such that $j(x)$ is rotationally symmetric and the integral $\int j(x)\,dx$ is equal to $1$. For every $\epsilon>0$ we define $j_\epsilon(x) := \epsilon^{-d} j(x/\epsilon)$, i.e.\ $j_\epsilon$ is again smooth, rotationally symmetric and satisfies $\supp j_\epsilon\subset \overline{B}_\epsilon(0)$ and $\int j_\epsilon(x)\,dx = 1$. For any locally bounded function $g:\rd \to \real$ the following convolution exists and defines a $C^\infty$-function:
\begin{gather*}
    g^\epsilon(x) := j_\epsilon * g(x) := \int g(x-y)j_\epsilon(y)\,dy,\quad x\in\rd.
\end{gather*}
Moreover, $\supp g^\epsilon \subset \supp g + \supp j_\epsilon \subset \supp g + \overline{B}_\epsilon(0)$. The function $g^\epsilon$ is called \emph{Friedrichs regularization} of $g$.

\subsection{Submultiplicative functions}
A function $g:\rd \to [0,\infty)$ is said to be \emph{submultiplicative} if there exists a constant $c = c(g)\in [1,\infty)$ such that
\begin{gather*}
    \forall x,y\in\rd\::\: g(x+y) \leq c g(x) g(y).
\end{gather*}
In order to avoid pathologies, we consider only measurable submultiplicative functions.\footnote{An example of a non-measurable submultiplicative function is $g(x) = \eup^{a(x)}$, $x\in\real$, where $a$ is a non-measurable solution to the functional equation $a(x+y)=a(x)+a(y)$.}
Every locally bounded submultiplicative function grows at most exponentially, i.e.\ there are constants $a,b\in (0,\infty)$ such that $g(x)\leq a \eup^{b|x|}$. Since $1+g$ inherits submultiplicativity from $g$, we may assume that $g\geq 1$. The following lemma shows that we can even assume that a submultiplicative function is smooth.
\begin{lemma}\label{intro-05}
    Let $g$ be a locally bounded submultiplicative function and $g^\epsilon$ its Fried\-richs regularization. Then $g^\epsilon\in C^\infty$ is submultiplicative and it satisfies
    \begin{gather}\label{intro-e16}
        \forall x\in\rd\::\quad c_\epsilon^{-1} g(x) \leq g^\epsilon(x)\leq c_\epsilon g(x)
    \end{gather}
    for some constant $c_\epsilon = c_{\epsilon,g}$.
\end{lemma}
\begin{proof}
    Submultiplicativity follows immediately from the two-sided estimate \eqref{intro-e16}:
    \begin{gather*}
        \smash[t]{g^\epsilon(x+y) \leq c_\epsilon g(x+y) \leq c_\epsilon c g(x)g(y) \leq c_\epsilon^3 c g^\epsilon(x) g^\epsilon(y).}
    \end{gather*}
    In order to see \eqref{intro-e16}, we use the definition of $g^\epsilon$ and fact that $g$ is submultiplicative,
    \begin{align*}
        g^\epsilon(x)
        &= \smash[b]{\int g(x-y)j_\epsilon(y)\,dy
        \leq c g(x) \int g(-y)j_\epsilon(y)\,dy
        \leq c \sup_{|y|\leq \epsilon} g(y) g(x)}
    \intertext{and}
        g(x)
        &= \int g(x)j_\epsilon(y)\,dy
        \leq c\int g(x-y)g(y)j_\epsilon(y)\,dy
        \leq c \sup_{|y|\leq \epsilon} g(y) g^\epsilon(x).
    \qedhere
    \end{align*}
\end{proof}

\section{Generalized moments and uniform integrability}\label{sec-gen}

Let $(X_t)_{t\geq 0}$ be a L\'evy process with triplet $(b,Q,\nu)$. The following moment result for a locally bounded submultiplicative functions $g$ is well-known, cf.\ Sato~\cite[Theorem 25.3, p. 159]{sato}:
\begin{gather}\label{gen-e18}
    \Ee \left[g(X_t)\right] < \infty \, \, \text{for some (hence, all) $t>0$}
    \iff
    \int_{|y|\geq 1} g(y)\,\nu(dy) < \infty.
\end{gather}
Our aim is to show that this is also equivalent to a certain uniform integrability condition.  Although we cast the statement and proof for L\'evy processes, an extension to certain L\'evy-type processes is possible; see Remark~\ref{gen-07} below. We denote by $\Tcal$ the family of stopping times for the process $X$ equipped with its natural filtration.

\begin{theorem}\label{gen-06}
    Let $(X_t)_{t\geq 0}$ be a L\'evy process with generator $A$, transition semigroup $(P_t)_{t\geq 0}$, and triplet $(b,Q,\nu)$, and let $g$ be a locally bounded submultiplicative function. The following assertions are equivalent:
    \begin{enumerate}\itemsep=6pt
    \item\label{gen-06-a}
        $\Ee \left[g(X_t)\right]$ is finite for some \textup{(}hence, all\textup{)} $t>0$;
    \item\label{gen-06-b}
        $\Ee \left[\sup_{s\leq t} g(X_s)\right]$ is finite for some \textup{(}hence, all\textup{)} $t>0$;
   \item\label{gen-06-c}
        $\left\{g(X_{\sigma})\right\}_{\sigma \in \Tcal, \sigma \leq t}$ is uniformly integrable for every $t>0$, i.e. 
        \begin{equation*}
	        \lim_{R\to\infty} \sup_{\sigma\in \Tcal,\sigma\leq t}\int_{ g(X_\sigma)>R}g(X_\sigma)\, d\Pp=0.
	       \end{equation*}
    \item\label{gen-06-d}
        $\sup_{\sigma\in\Tcal, \sigma\leq t}\Ee \left[g(X_\sigma)\right]$ is finite for every $t>0$;
    \item\label{gen-06-e}
        $\int_{|y|\geq 1} g(y)\,\nu(dy) < \infty$;
    \item\label{gen-06-f}
        The adjoint semigroup $P^*_tf(x):=\Ee \left[f(x-X_t)\right]$ is a strongly continuous semigroup on the weighted $L^1$-space $L^1(\rd,g)$.
    \item\label{gen-06-g}
    	The adjoint generator $(A^* \phi)(x) := (A \phi)(-x)$ satisfies $A^* \phi \in L^1(\real^d,g)$ for all $\phi \in C_c^{\infty}(\real^d)$.
	\item\label{gen-06-h}
		There exists a non-negative $\phi \in C_c^{\infty}(\rd)$, $\phi \not \equiv 0$, with $A^* \phi \in L^1(\rd,g)$.
    \end{enumerate}
    If one \textup{(}hence, all\textup{)} of the conditions is satisfied, then there are constants $c_i>0$, $i=1,2$, such that
    \begin{align}\label{eq-exp-1}
    	\Ee[g(X_t)] \leq c_1 \eup^{c_2 t}, \qquad t \geq 0,
    \end{align}
    and $\|A^* \phi\|_{L^1(\rd,g)}$ is bounded by a constant multiple of
    \begin{align*}
    	 \left( |b|_{\ell^1} + ||Q||_{\ell^1} + \int_{y \neq 0} (1 \wedge |y|^2) \, \nu(dy) + \int_{|y|\geq1} g(y) \, \nu(dy) \right) \|\phi\|_{C_b^2(\rd)}.
    \end{align*}
\end{theorem}

\begin{remark}\label{gen-07}
    There is another equivalent condition if $x\mapsto g(|x|)$ is locally bounded, submultiplicative and $g(r)$ is increasing:
    
    \begin{enumerate}\setcounter{enumi}{8}
    \item\label{gen-06-i}
    $\Ee\left[g\left(\sup_{s\leq t}|X_s|\right)\right]<\infty$ for some \textup{(}hence, all\textup{)} $t>0$.
    \end{enumerate}
    
    \noindent
    The direction \ref{gen-06-i}$\Rightarrow$\ref{gen-06-a} follows from the assumption that $g$ is increasing. The implication \ref{gen-06-b}$\Rightarrow$\ref{gen-06-i} does not need monotonicity since we have
    \begin{gather*}
        g\left(\sup\nolimits_{s\leq t}|X_s|\right) \leq \sup\nolimits_{s\leq t} g\left(|X_s|\right)
    \end{gather*}
    at least if $x\mapsto g(|x|)$ is continuous. This can always be achieved by a Friedrichs regularization.

    Let us also point out that we may replace $A^*$ and $P_t^*$ by $A$, and $P_t$ if either \textup{(}the law of\textup{)} $X_t$ is symmetric or if $g$ is even, i.e.\ $g(x)=g(-x)$.
\end{remark}

\begin{remark}\label{gen-077}
    Many of the equivalent criteria in Theorem~\ref{gen-06} are scattered throughout the analysis and probability literature, and they are usually proved separately, using ad-hoc methods.

    The equivalence of \ref{gen-06-a}, \ref{gen-06-d} (for deterministic times), \ref{gen-06-e} and \ref{gen-06-i} can be found in Sato \cite[Chapter~25]{sato}. The equivalence of \ref{gen-06-a} and \ref{gen-06-b} is due to Siebert \cite{siebert}, and variants of \ref{gen-06-g}, \ref{gen-06-h} appear first in Hulanicki \cite{hulanicki}; their proofs are cast in the language of probability on \textup{(}Lie\textup{)} groups. The conditions \ref{gen-06-c}, \ref{gen-06-f}, and the condition \ref{gen-06-d} \textup{(}with bounded stopping times\textup{)} are new, and so are the streamlined proofs given in the present paper.

    Some of our arguments carry over to L\'evy-type processes whose generators have bounded coefficients \textup{(}see \cite[p.~55]{BSW} for the notation\textup{)}; in particular
    \ref{gen-06-e}$\Rightarrow$\ref{gen-06-a} \textup{[}using the alternative proof below\textup{]}
    $\Rightarrow$\ref{gen-06-b}$\Rightarrow$\ref{gen-06-c}$\Rightarrow$\ref{gen-06-d} becomes
     \begin{align*}
     	\sup_{x \in \rd} \int_{|y| \geq 1} g(y) \, \nu(x,dy)<\infty
     	&\implies \sup_{x \in \rd} \Ee^x[g(X_t-x)]<\infty\\
     	&\implies \sup_{x \in \rd} \Ee^x[\sup_{s \leq t} g(X_s-x)]<\infty \\
     	&\implies  \{g(X_{\sigma}-x)\}_{\sigma \in \Tcal, \sigma \leq t} \; \text{is uniformly\ integrable} \\
     	&\implies \sup_{\sigma \in \Tcal,\sigma \leq t} \Ee^x[g(X_{\sigma}-x)]<\infty,
     \end{align*}
     while
     \ref{gen-06-d}$\Rightarrow$\ref{gen-06-e} only yields $\inf_{x\in\rd} \int g(y) \, \nu(x,dy)<\infty$, and an additional condition of the type
     \begin{gather*}
        \sup_{x\in\rd} \int_{|y| \geq 1} g(y) \, \nu(x,dy) \leq C \inf_{x\in\rd} \int_{|y| \geq 1} g(y) \, \nu(x,dy)
     \end{gather*}
     is needed to get equivalences; this is partly worked out in \cite{kuehn-spa}.
\end{remark}

In order to prove this theorem, we need a few preparations.

\begin{lemma}\label{gen-08}
    Let $g$ be a locally bounded submultiplicative function. If $(X_t)_{t\geq 0}$ is a L\'evy process, then
    \begin{gather*}
        \Ee\left[\sup_{s\leq T} g(X_s)\right] = \kappa_T < \infty
    \quad\text{implies}\quad
        \Ee\left[\sup_{s\leq 2T} g(X_s)\right] \leq \kappa_T(1+c\kappa_T)<\infty
    \end{gather*}
    and
    \begin{equation*}
    	\Ee[g(X_t)] < \infty \, \, \text{for some $t>0$}
    	    \quad\text{implies}\quad
    	\Ee[g(X_t)]<\infty \, \, \text{for all $t>0$}.
    \end{equation*}
\end{lemma}
\begin{proof}
   \textbf{\primo}
    We have 
    \begin{align*}
        \Ee\left[\sup_{s\leq 2T} g(X_s)\right]
        \leq& \Ee\left[\sup_{s\leq T} g(X_s)\right] + \Ee\left[\sup_{T \leq s\leq 2T} g(X_s)\right]\\
        =& \kappa_T + \Ee\left[\sup_{s\leq T} g(X_{s+T})\right].
    \end{align*}
    Since $g$ is a submultiplicative function, we see $g(X_{s+T}) \leq c g(X_{s+T}-X_T)g(X_T)$; moreover, $X_T$ and the process $(X_{s+T}-X_T)_{s\geq 0}\sim (X_s)_{s\geq 0}$ are independent,\footnote{We may replace this by the (strong) Markov property if $\sup\limits_{y\in\rd} \sup\limits_{s\leq T} \Ee^y[g(X_s-y)]<\infty$. For L\'evy processes the first supremum is always trivial.} and so
    \begin{align*}
        \Ee\left[\sup_{s\leq T} g(X_{s+T})\right]
        &\leq c \smash[t]{\Ee\left[\sup_{s\leq T} g(X_{s+T}-X_T)\right] \Ee\left[ g(X_T)\right]}\\
        &\leq c \Ee\left[\sup_{s\leq T} g(X_{s})\right] \Ee\left[ g(X_T)\right]
        \leq c\kappa_T^2.
    \end{align*}

    \medskip\textbf{\secundo} Let $t_0>0$ such that $\Ee[g(X_{t_0})]<\infty$. Using the Markov property we see that for any $s <t_0$
    \begin{gather*}
        \Ee\left[g(X_{t_0})\right]
        = \Ee\left[g(X_{t_0}-X_s + X_s)\right]
        = \smash[t]{\int_\rd \Ee\left[g(X_s+y)\right] \Pp(X_{t_0-s}\in dy).}
    \end{gather*}
    Thus, there is some $y$ such that $\Ee\left[g(X_s+y)\right]<\infty$, and we conclude from the submultiplicative property that $\Ee\left[g(X_s)\right] \leq cg(-y)\Ee\left[g(X_s+y)\right]<\infty$ for all $s\leq t_0$.  As before, we can now show that $\Ee \left[g(X_{2t_0})\right]<\infty$ and, by iteration, we see that $\Ee\left[g(X_t)\right]<\infty$ for all $t>0$.
\end{proof}

\begin{lemma}\label{gen-09}
    Let $g$ be a locally bounded submultiplicative function and denote by $g^\epsilon$ its regularization with a Friedrichs mollifier. If $A$ is the generator of a L\'evy process given by \eqref{intro-e02}, then $|A g^\epsilon(x)|$ is bounded by
    \begin{gather*}
     \left(  C_\epsilon \cdot C(b,Q,\nu;g)\cdot \sup_{|y|\leq 1}g(y) \right) \cdot g(x),
    \end{gather*}
    where the constant $C(b,Q,\nu;g)$ is of the form
    \begin{gather*}
        C(b,Q,\nu;g)
        =
        \left(|b|_{\ell^1}+|Q|_{\ell^1}+ \int_{y\neq 0} \left(1\wedge|y|^2\right) \nu(dy) + \int_{|y|\geq 1} g(y)\,\nu(dy) \right).
    \end{gather*}
\end{lemma}
Note that the constant $C_\epsilon$ appearing in Lemma~\ref{gen-09} is, in general, unbounded as $\epsilon\to 0^+$.
\begin{proof}
    Without loss of generality we can assume that $g\geq 1$. Otherwise we use $g+1$ instead of $g$ and observe that $A(g+1)^\epsilon = A(g^\epsilon + 1) = A g^\epsilon$. We set $s_\epsilon:=\sup_{|y|\leq \epsilon}g(y)$ and as in \eqref{intro-e02} write $A = L + J + K$.

    Observe that $|\partial^\alpha g^\epsilon(x)| \leq c_{\alpha,\epsilon}s_\epsilon g(x)$ holds for every multi-index $\alpha\in\nat_0^d$. This follows from $\partial^\alpha g^\epsilon(x) = (\partial^\alpha j_\epsilon)*g(x)$ and
    \begin{align*}
        \left|\partial^\alpha g^\epsilon(x)\right|
        \leq \int \left|\partial^\alpha j_\epsilon(y)\right| g(x-y)\,dy
        \leq cg(x) s_\epsilon \int \left|\partial^\alpha j_\epsilon(y)\right| dy.
    \end{align*}

    We can now estimate the three parts of $A$ separately. For the local part we use the above estimate with $|\alpha|=1$ and $|\alpha|=2$:
    \begin{align*}
        \left|L g^\epsilon\right|(x)
        &\leq s_\epsilon  \Bigg(\sum_{i=1}^d |b_i| c_{\epsilon,i} + \frac 12 \sum_{i,k=1}^d |q_{ik}| c_{\epsilon,i,k}\Bigg)  \cdot g(x)\\
        &\leq c_\epsilon s_\epsilon\left(|b|_{\ell^1}+|Q|_{\ell^1}\right)\cdot g(x).
    \end{align*}

\noindent
    The large-jump part is estimated using $|\alpha|=0$ and the sub\-multi\-pli\-ca\-ti\-vi\-ty of $g$:
    \begin{align*}
        \left|K g^\epsilon(x)\right|
        &\leq c_\epsilon s_\epsilon\int_{|y|\geq 1} \left(g(x) g(y) + g(x)\right) \nu(dy)\\
        &= c_\epsilon s_\epsilon \int_{|y|\geq 1} \left(g(y) + 1\right) \nu(dy)\cdot g(x).
    \end{align*}

    Using Taylor's formula with integral remainder term we can rewrite the part containing the small jumps and we see that
    \begin{align*}
        \left|J g^\epsilon(x)\right|
        &= \left|\sum_{i,k=1}^d \int_{0<|y|<1}\int_0^1  \partial_i\partial_k g^\epsilon(x+ty) y_i y_k (1-t)\,dt\,\nu(dy).\right|\\
        &\leq c_\epsilon s_\epsilon \sum_{i,k=1}^d c_{\epsilon,i,k} \int_{0<|y|<1}\int_0^1 g(x+ty)|y_i y_k|(1-t)\,dt\,\nu(dy)\\
        &\leq c_\epsilon s_\epsilon \sum_{i,k=1}^d c_{\epsilon,i,k} \int_{0<|y|<1}\int_0^1 g(ty) |y_i y_k|(1-t)\,dt\,\nu(dy)\cdot g(x)\\
        &\leq c_\epsilon' \sup_{|y|\leq 1} g(y) \int_{0<|y| <1}|y|^2\,\nu(dy)\cdot g(x).
    \end{align*}

    If we combine these three estimates, the claim follows.
\end{proof}

\begin{proof}[Theorem~\ref{gen-06}]
We show \ref{gen-06-a}$\Rightarrow$\ref{gen-06-b}$\Rightarrow$\ref{gen-06-c}$\Rightarrow$\ref{gen-06-d}$\Rightarrow$\ref{gen-06-e}$\Rightarrow$\ref{gen-06-g}$\Rightarrow$\ref{gen-06-h}$\Rightarrow$\ref{gen-06-a} and then \ref{gen-06-c}$\Rightarrow$\ref{gen-06-f}$\Rightarrow$\ref{gen-06-a}.\footnote{A direct proof of \ref{gen-06-e}$\Rightarrow$\ref{gen-06-a} is given in the next section. This is particularly interesting, if we want to go beyond L\'evy processes.}
Throughout the proof we will assume that $g\geq 1$ and $g\in C^2(\rd)$. Otherwise we could replace $g$ by its Friedrichs regularization $g^{\epsilon}$, see Lemma~\ref{intro-05}, and $g+1$, resp., $g^{\epsilon}+1$.

\medskip\textbf{\primo}
$\ref{gen-06-a}\Rightarrow\ref{gen-06-b}$. If $\Ee[g(X_t)]$ is finite for some $t>0$, then Lemma~\ref{gen-08} shows that $\Ee[g(X_t)]$ is finite for all $t>0$. For $a,b> 0$ we define a stopping time
\begin{gather*}
    \sigma := \sigma_{a,b} := \inf\left\{s\mid g(X_s) > cg(0) \eup^{a+b}\right\}
\end{gather*}
and observe that we can use the submultiplicativity of $g$ to get
\begin{align*}
    \Pp\left( g(X_t) > \eup^a\right)
    &\geq \Pp\left( g(X_t) > \eup^a,\; g(X_\sigma)\geq  cg(0)\eup^{a+b},\; \sigma \leq t\right)\\
    &\geq \Pp\left( g(X_\sigma - X_t) < g(0)\eup^b,\; g(X_\sigma)\geq  cg(0)\eup^{a+b},\; \sigma \leq t\right).
\end{align*}
The strong Markov property yields (for all $t\leq T$, $T$ will be determined in the following step)
\begin{align} \label{gen-e10} \begin{aligned}
    \Pp\left( g(X_t) > \eup^a\right)
    &\geq \int_{\sigma \leq t} \Pp\left( g(-X_{t-\sigma(\omega)}) < g(0)\eup^b\right) \Pp(d\omega)\\
    &\geq \inf_{r\leq t} \Pp\left( g(-X_{r}) < g(0)\eup^b\right)\cdot \Pp(\sigma \leq t)\\
    &\geq \frac 12 \Pp\left(\sup\nolimits_{s \leq t} g(X_s) > cg(0) \eup^{a+b}\right).
    \end{aligned}
\end{align}
In the last estimate we use that $\left\{\sup\nolimits_{s\leq t} g(X_s)>  cg(0) \eup^{a+b}\right\}\subseteq\{\sigma \leq t\} $. The factor $\frac 12$ comes from the fact that $g$ is locally bounded and $\lim_{t\to 0^+}\Pp(|X_t|>\epsilon)=0$ (continuity in probability), which shows that there is some $0<T\leq t_0$ such that
\begin{gather*}
    \Pp\left(g(-X_t)< g(0)\eup^b\right) \geq \frac 12
    \quad\text{for all $t\leq T$}.
\end{gather*}
This proves that for $\eup^\gamma := cg(0)\eup^{b}$
\begin{gather*}
    \Pp\left(\sup\nolimits_{s\leq T} g(X_s) > \eup^{a+\gamma}\right) \leq 2\Pp(g(X_T)>\eup^a).
\end{gather*}
We can now use the layer-cake formula to see that
\begin{align}
\begin{split}\label{gen-e15}
    \Ee\left[\sup_{s\leq T}g(X_s)\right]
    &= 1+\int_1^\infty \Pp\left(\sup\nolimits_{s\leq T} g(X_s)>y\right)\,dy\\
    &= 1+\eup^\gamma\int_{-\gamma}^\infty \Pp\left(\sup\nolimits_{s\leq T} g(X_s)>\eup^\gamma \eup^a\right)\eup^a\,da\\
    &\leq 1+\eup^\gamma\int_{-\gamma}^0 \eup^a\,da + 2\eup^\gamma \int_0^\infty \Pp\left(g(X_T)> \eup^a\right)\eup^a\,da\\
    &\leq \eup^\gamma + 2\eup^\gamma \Ee\left[g(X_T)\right].
\end{split}
\end{align}
Using Lemma~\ref{gen-08} we see that $\Ee\left[\sup_{s\leq n T}g(X_s)\right]<\infty$ for all $n\in\nat$, i.e.\ \ref{gen-06-b} holds for all $t>0$.

\medskip\textbf{\secundo}
$\ref{gen-06-b}\Rightarrow\ref{gen-06-c}$. If $\sup_{s \leq t} g(X_s)$ is integrable for some $t>0$, then it is integrable for all $t>0$, cf.\ Lemma~\ref{gen-08}. For fixed $t>0$, let  $\sigma\in\Tcal$ be a stopping time with $\sigma\leq t$. Then we have
\begin{gather*}
    g(X_\sigma) \leq \sup_{r\leq t} g(X_r) \in L^1(\Pp).
\end{gather*}
Consequently, the family $\{g(X_{\sigma})\}_{\sigma \in \Tcal,\sigma \leq t}$ is dominated by the integrable random variable $\sup_{r \leq t} g(X_r)$; hence, it is uniformly integrable.

\medskip\textbf{\tertio}
$\ref{gen-06-c}\Rightarrow\ref{gen-06-d}$. This is immediate from the fact that uniform integrability implies boundedness in $L^1$.

\medskip\textbf{\quarto}
$\ref{gen-06-d}\Rightarrow\ref{gen-06-e}$. Recall the definition of $\tau_R$ from Lemma \ref{intro-03}. We rearrange \eqref{intro-e12} and insert it into \eqref{intro-e10} to get
\begin{align*}
    &\Ee\left[(g\chi_R)(X_{t\wedge\tau_R})\right]
    - \Ee\left[\int_{[0,t\wedge\tau_R)} (L+J)g(X_s)\,ds\right]\\
    &\qquad =  g(0) +  \Ee\left[\int_{[0,t\wedge\tau_R)} \int_{|y|\geq 1} \left(g\chi_R(X_s+y) - g(X_s)\right)\nu(dy)\,ds\right]\\
    &\qquad =  g(0) +  \Ee\left[\int_{[0,t\wedge\tau_R)} \int_{|y|\geq 1} g\chi_R(X_s+y)\,\nu(dy)\,ds\right] \\
    &\qquad\qquad\mbox{}- \nu(|y|\geq 1)\cdot \Ee\left[\int_{[0,t\wedge\tau_R)}g(X_s)\,ds\right].
\end{align*}
Now we use Lemma~\ref{gen-09} for the L\'evy generator $L+J$ and the estimates
\begin{gather*}
    \smash[b]{\Ee\left[\int_{[0,t\wedge\tau_R)}g(X_s)\,ds\right]
    \leq \int_{[0,t)} \Ee \left[g(X_s)\right] ds
    \leq t \sup_{s\leq t}\Ee \left[g(X_s)\right]}
\intertext{and}
    \smash[t]{\Ee\left[(g\chi_R)(X_{t\wedge\tau_R})\right]
    \leq \Ee\left[g(X_{t\wedge\tau_R})\right]
    \leq \sup_{\sigma\in\Tcal, \sigma\leq t} \Ee\left[g(X_\sigma)\right].}
\end{gather*}
The constant appearing in Lemma \ref{gen-09} for $L+J$ depends only on $\nu|_{B_1(0)}$, in particular it is independent of $\int_{|y|>1} g(y)\nu(dy)$.
Because of our assumption~\ref{gen-06-d}, there is a constant $C_t$, not depending on $R$ or $\int_{|y|>1} g(y)\nu(dy)$,  such that
\begin{align*}
    C_t \geq \Ee\left[\int_{[0,t\wedge\tau_R)} \int_{|y|\geq 1} g\chi_R(X_s+y)\,\nu(dy)\,ds\right].
\end{align*}
Letting $R\to\infty$, Fatou's lemma and yet another application of submultiplicativity yield
\begin{align*}
    C_t
    &\geq \Ee\left[\int_{[0,t)} \int_{|y|\geq 1} g(X_s+y)\,\nu(dy)\,ds\right]\\
    &\geq \frac 1c \Ee\left[\int_{[0,t)} \int_{|y|\geq 1} \frac{g(y)}{g(-X_s)}\,\nu(dy)\,ds\right]\\
    &\geq \frac 1c \Ee \left( \int  \I_{\{|X_s|\leq\delta\}} \,\frac{ds}{g(-X_s)}\right) \int_{|y|\geq 1} g(y)\,\nu(dy)
\end{align*}
for any $\delta\in(0,1)$. Since $g$ is locally bounded, $g\ge 1$, and $s\mapsto X_s$ is c\`adl\`ag, \ref{gen-06-e} follows.

\medskip\textbf{\quinto}
	$\ref{gen-06-e}\Rightarrow\ref{gen-06-g}$. Let $\phi \in C_c^{\infty}(\real^d)$. The reflection $\widetilde{\phi}(x) := \phi(-x)$ is again in $C_c^{\infty}(\rd)$ and thus $\|A^* \widetilde{\phi}\|_{\infty}<\infty$ by the representation of $A^*$ as an integro-dif\-fer\-en\-tial operator \eqref{intro-e02} and Taylor's formula. Choose $R>0$ such that $\supp \phi$ is contained in the ball $B_R(0)$, then
	\begin{equation*}
		(A^* \widetilde{\phi})(-x) = \int_{y \neq 0} \phi(x-y) \, \nu(dy), \quad |x| \geq 2R.
	\end{equation*}
	Since $g$ is bounded on compact sets, it follows that
	\begin{align*}
		&\int_{\rd} g(x) |(A^* \widetilde{\phi})(x)| \, dx
		= \left( \int_{|x| \leq 2R} + \int_{|x|>2R} \right)  g(x) \, |(A^* \widetilde{\phi})(x)| \, dx \\
		&\quad\leq |B_{2R}(0)| \sup_{|x| \leq 2R} g(x) \|A^* \widetilde{\phi}\|_{\infty} + \int_{|x|>2R} \int_{y \neq 0} g(x) |\phi(x-y)| \, \nu(dy) \, dx.
	\end{align*}
	It remains to show that the integral expression on the right-hand side is finite. By Tonelli's theorem and a change of variables ($z=x-y$),
	\begin{align*}
		I:=& \int_{|x|>2R} \int_{y \neq 0} g(x) |\phi(x-y)| \, \nu(dy) \, dx\\
		=& \int \int \I_{|z+y|>2R} \I_{|z| \leq R} |\phi(z)| g(z+y) \, dz \, \nu(dy),
	\end{align*}
	where we use that $\supp \widetilde{\phi} \subseteq B_R(0)$. The estimate $\I_{|z+y|>2R} \I_{|z| \leq R} \leq \I_{|y| \geq R}$ and the submultiplicativity of $g$ now yield
	\begin{align*}
		I \leq c \left( \int_{|z| \leq R} |\phi(z)| \, g(z) \, dz \right) \left( \int_{|y| \geq R} g(y) \, \nu(dy) \right)<\infty;
	\end{align*}
	the integrals are finite because of \ref{gen-06-e} and the local boundedness of $g$, $\phi$.

\medskip\textbf{\sexto}
	$\ref{gen-06-g}\Rightarrow\ref{gen-06-h}$. Trivial.
	
\medskip\textbf{\septimo}
	$\ref{gen-06-h}\Rightarrow\ref{gen-06-a}$.	This part of the proof draws from a work by Hulanicki \cite{hulanicki}. Take $\phi \in C_c^{\infty}(\rd)$ non-negative such that $\phi \not \equiv 0$ and $A^* \phi \in L^1(\rd,g)$. Let us first assume that $g$ is bounded. Set
	\begin{equation*}
		h(t):= \int_{\rd} (P_t \phi)(-x) g(x) \, dx, \quad t \geq 0,
	\end{equation*}
	where $(P_t u)(x) = \Ee[u(x+X_t)]$ is the semigroup. We  want 
    that $|h'(t)| \leq C h(t)$ for some constant $C>0$. An application of Tonelli's theorem and a change of variables yield
	\begin{align*}
		\int_{\rd} |P_t A \phi(-x)| g(x) \, dx
		&\leq \Ee \left[ \int_{\rd} |(A \phi)(-y)| \, g(y+X_t) \, dy \right] \\
		&\leq c \Ee \left[ g(X_t) \right] \int_{\rd} |(A^* \phi)(y)| g(y) \, dy < \infty;
	\end{align*}
	the latter integral is finite because $A^* \phi \in L^1(\rd,g)$ and $g$ is bounded. Moreover, Dynkin's formula \eqref{intro-e06} entails that $\frac{d}{dt} P_t \phi = P_t A \phi$. Consequently, we can use a version of the differentiation lemma for parameter-dependent integrals, cf.\ \cite[Proposition A.1]{kuehn-schilling19},\footnote{This is a variant of the classical differentiation lemma for parameter-dependent integrals. If $u=u(t,x)$ is a function such that $t \mapsto u(t,x)$ is differentiable, $\int_a^b \int_{\rd} |\partial_t u(t,x)| \, dx \, dt < \infty$ and $t \mapsto \int_{\rd} \partial_t u(t,x) \, dx$ is continuous, then $t \mapsto \int_{\rd} u(t,x) \, dx$ is differentiable on $(a,b)$ with derivative $\int_{\rd} \partial_t u(t,x)\, dx$. These conditions are clearly satisfied in our application, see the calculation immediately before the present paragraph.} to obtain
	\begin{equation*}
		h'(t)
		= \int_{\rd} \frac{d}{dt} (P_t \phi)(-x) g(x) \, dx
		= \int_{\rd} (P_t A \phi)(-x) g(x) \, dx
	\end{equation*}
	and, by the above estimate,
	\begin{align*}
		|h'(t)|
		\leq c \Ee \left[ g(X_t) \right] \int_{\rd} |(A^* \phi)(y)| g(y) \, dy.
	\end{align*}
	The submultiplicativity of $g$ gives
	\begin{equation*}
		\left( \int_{\rd} \frac{\phi(x)}{g(x)} \, dx \right) g(y)
		\leq c \int_{\rd} \phi(x) g(y-x) \, dx
		= c (\phi*g)(y)
	\end{equation*}
	for all $y \in \rd$, i.e.\
	\begin{equation}
		g(y) \leq \frac{c}{\|\phi\|_{L^1(\rd,1/g)}} (\phi*g)(y); \label{gen-e30}
	\end{equation}
	note that
	\begin{equation*}
		\|\phi\|_{L^1(\rd,1/g)} = \int_{\rd} \frac{\phi(x)}{g(x)} \, dx \in (0,\infty)
	\end{equation*}
	because $g \geq 1$ is locally bounded and $\phi>0$ on a set of positive Lebesgue measure.  Using \eqref{gen-e30} for $y=X_t$, we get
	\begin{equation*}
		|h'(t)|
		\leq C \Ee[(\phi*g)(X_t)]
		= C \int_{\rd} \Ee[\phi(X_t-x)] g(x) \, dx
		= C h(t),\quad t\geq 0.
	\end{equation*}
	Hence, by Gronwall's lemma,
	\begin{equation*}
		h(t) \leq h(0) \eup^{\alpha t}, \quad t \geq 0,
	\end{equation*}
	for some constant $\alpha > 0$. Invoking once more \eqref{gen-e30}, we conclude that
	\begin{align}\label{eq-exp-2}
		\Ee[g(X_t)]
		\leq c' \Ee[(\phi*g)(X_t)]
		= c' h(t)
		\leq c' h(0) \eup^{\alpha t}.
	\end{align}
	So far, we assumed that $g$ is bounded. For unbounded $g$, we replace $g$ by $\min\{g,n\}$ -- which is again submultiplicative -- in the above estimates and find that
	\begin{equation*}
		\Ee[\min\{g(X_t),n\}] \leq c'' \eup^{\alpha t}
	\end{equation*}
	for some constants $c''>0$, $\alpha>0$, not depending on $n \in \nat$. Thus, by Fatou's lemma, it holds that
	$\Ee[g(X_t)] \leq c'' \eup^{\alpha t}$ for all $t \geq 0$.

\medskip\textbf{\octavo}
	$\ref{gen-06-c}\Rightarrow\ref{gen-06-f}$. Let $f\in L^1(\rd,g)$. We see that
	\begin{align*}
		\int_{\rd} |P^*_tf(x)|\, g(x) dx
		&\leq \Ee \left[\int_{\rd} |f(x-X_t)|\, g(x-X_t+X_t)dx\right]\\
	    &\leq c \Ee \left[g(X_t)\right] \| f\|_{L^1(\rd,g)},
	\end{align*}
	which shows that $P^*_t :L^1(\rd,g)\to L^1(\rd,g)$ is continuous. Let $\phi \in C_c(\rd)$ and assume that $\supp \phi$ is contained in some ball $B_R(0)$ with radius $R>0$. We show that $P^*_t\phi\to \phi$ in $L^1(\rd,g)$ as $t\to 0^+$. Since $\phi$ is uniformly continuous, we can pick $\epsilon=\epsilon(\eta)>0$ in such a way that $|\phi(x+y)-\phi(x)|\leq\eta$ for all $x$ and all $|y|<\epsilon$. Thus,
	\begin{align*}
		\int_{\rd}& |P^*_t\phi(x)-\phi(x)|\, g(x)dx
	    \leq  \Ee \left[\int_{\rd} |\phi(x-X_t)-\phi(x)|\, g(x) dx\right] \\
	     & = \Ee \left[\int_{ B_{R+\epsilon}(0)} |\phi(x-X_t)-\phi(x)| \I_{|X_t|<\epsilon}\, g(x) dx\right]\\
	     &\qquad\mbox{} + \int_{\rd} \Ee \left[|\phi(x-X_t)-\phi(x)| \I_{|X_t|\geq\epsilon}\right] g(x) dx \\
	     &\leq |B_{R+\epsilon}(0)|\sup_{|x| < R+\epsilon} g(x)\cdot \eta
	     + \Pp(|X_t|\geq\epsilon) \|\phi\|_{L^1(\rd,g)}\\
	     &\qquad\mbox{}+ c\,\Ee\left[g(X_t)\I_{|X_t|\geq\epsilon}\right] \|\phi\|_{L^1(\rd,g)}.
	\end{align*}
	Since the family $(g(X_t))_{t \leq 1}$ is uniformly integrable, $X_t\to 0$ in probability, and $\eta>0$ is arbitrary, we obtain that $\|P^*_t\phi-\phi\|_{L^1(\rd,g)}\to 0$ as $t\to 0^+$. Using that $C_c(\rd)$ is a dense subset of $L^1(\rd,g)$, we conclude that $(P^*_t)_{t\geq 0}$ is a strongly continuous operator semigroup on $L^1(\rd,g)$.	

\medskip\textbf{\nono}
    $\ref{gen-06-f}\Rightarrow\ref{gen-06-a}$.  Let $\phi\in C_c(\rd)\cap L^1(\rd,g)$ such that $\phi\geq 0$ and $\phi=1$ on $B_1(0)$. Using $g\geq 1$ and the submultiplicative property of $g$, we see that
    \begin{gather*}
        \|P^*_t\phi\|_{L^1(\rd,g)}
        = \Ee \left[\int_{\rd}  (g(x+X_t))\phi(x)dx\right]
        \geq \frac 1c \Ee\left[g(X_t)\right] \int_{\rd} \frac{\phi(x)\,dx}{ g(-x)};
    \end{gather*}
    this implies that $\Ee\left[ g(X_t)\right] <\infty$.

\medskip
    Since all conditions are equivalent, the exponential bound \eqref{eq-exp-1} follows from Gronwall's inequality, see \eqref{eq-exp-2} in Step~\septimo.
\end{proof}

The proof of Theorem~\ref{gen-06} contains the following moment result for L\'evy processes with bounded jumps. Alternate proofs can be found in
Sato \cite[Theorem 25.3, p.~159]{sato} or \cite[Lemma~8.2]{barca}. If we use in Step~\tertio\ the submultiplicative function $g(x):= \eup^{\beta |x|}$, $\beta>0$, $x\in\rd$, we get the following corollary.
\begin{corollary}\label{gen-11}
    Let $Y=(Y_t)_{t\geq 0}$ be a L\'evy process with bounded jumps, i.e.\ the L\'evy measure has bounded support. Then $Y$ has exponential moments, i.e.\ $\Ee\left[ \eup^{\beta |Y_t|}\right] < \infty$ for all $\beta>0$ and $t\geq 0$.
\end{corollary}

If $Y$ has a non-degenerate jump part, then moments of the type $\Ee \left[\eup^{\beta |X_t|^{1+\epsilon}}\right]$ do not exist, cf.\ \cite[Theorem 3.3(c)]{den-sch15}.

\section{Doob's condition \eqref{DL} for L\'evy processes}\label{sec-doo}

Let $Y$ be a stochastic process and $\Tcal$ be the family of all stopping times with respect to the natural filtration of $Y$. Recall that $Y$ \emph{satisfies the condition} \eqref{DL} if for each fixed $t>0$ the family $(Y_{t\wedge\sigma})_{\sigma\in\Tcal}$ is uniformly integrable, i.e.\ if
\begin{gather}\tag{\textup{DL}}\label{DL}
    \forall t>0\::\quad \lim_{R\to\infty} \sup_{\sigma\in\Tcal} \int |Y_{t\wedge\sigma}| \I_{|Y_{t\wedge\sigma}| \geq R}\,d\Pp = 0.
\end{gather}
It is well known, cf.\ \cite[Proposition IV.1.7, p.~124]{revuz-yor}, that a local martingale is a martingale if, and only if, it is of class DL.

As a direct consequence of Theorem~\ref{gen-06}, we get the following characterization of the condition \eqref{DL} for functions of a L\'evy process.

\begin{corollary}\label{doo-20}
    Let $X = (X_t)_{t\geq 0}$ be a L\'evy process with triplet $(b,Q,\nu)$ and $g$ a locally bounded submultiplicative function. The following are equivalent.
    \begin{enumerate}\itemsep=4pt
    	\item\label{doo-20-a} The process $(g(X_t))_{t\geq 0}$ satisfies the condition \eqref{DL};
    	\item\label{doo-20-b} $\Ee[g(X_t)]$ is finite for some $t>0$;
    	\item\label{doo-20-c} $\int_{|y|\geq 1} g(y)\,\nu(dy)<\infty$.
    \end{enumerate}
\end{corollary}

\begin{corollary}\label{doo-21}
    Let $X$ be a L\'evy process and $f:\rd\to\real$ such that for some locally bounded submultiplicative function $g$ it holds $|f(x)|\leq g(x)$. If the expectation $\Ee \left[g(X_t)\right]$ is finite and $(f(X_t))_{t\geq 0}$ is a local martingale, then $(f(X_t))_{t\geq 0}$ is a martingale.
\end{corollary}
\begin{proof}
    Corollary~\ref{doo-20} shows that $(g(X_t))_{t\geq 0}$, hence $(f(X_t))_{t\geq 0}$ enjoys the property \eqref{DL}; therefore the local martingale $(f(X_t))_{t\geq 0}$ is already a proper martingale.
\end{proof}
The setting of Corollary~\ref{doo-21} is quite natural if one thinks of It\^o's formula or the expression \eqref{intro-e08} appearing in the discussion of Dynkin's formula: If $A$ is the (pointwise extension to positive $C^2$-functions of the) generator of $X$, and if $Af$ is equal to $0$, then $f(X_t)$ is, by It\^o's formula, a local martingale. Corollary~\ref{doo-21} thus gives a condition when this local martingale is a true martingale.

\subsection{A direct proof that Corollary~\ref{doo-20}.\ref{doo-20-c} entails \ref{doo-20-b}}
Sometimes it is useful to have a direct proof that existence of the moments of the L\'evy measure give generalized moments for the process. The approach below gives a method using standard `household' techniques from any course on Markov processes, notably Dynkin's formula and Gronwall's inequality; therefore it applies to more general (strong) Markov processes.

\begin{proof}[Alternative for Corollary~\ref{doo-20}.\ref{doo-20-c}$\Rightarrow\ref{doo-20-b}$ resp.\  Theorem~\ref{gen-06}.\ref{gen-06-e}$\Rightarrow\ref{gen-06-a}$]

In view of Lemma~\ref{intro-05} we may replace $g$ by its regularization $g^\epsilon$. Combining Dynkin's inequality (Lemma~\ref{intro-03}) and the estimate from Lemma~\ref{gen-09} shows
\begin{align*}
    \Ee\left[ g^\epsilon(X_{t\wedge\sigma})\right]
    \leq g^\epsilon(0) + C_{\epsilon,g}^{b,Q,\nu}\Ee \left[\int_{[0,t\wedge\sigma)} g(X_s)\,ds \right].
\end{align*}
The constant $C$ depends on $\epsilon$, the triplet $(b,Q,\nu)$ and on $\int_{|y|\geq 1} g(y)\,\nu(dy)$, see Lemma~\ref{gen-09}.
If we replace $\sigma$ by  $\sigma\wedge\tau_R$ with $\tau_R = \inf\{s\geq 0 \mid |X_s|\geq R\}$ and set $\kappa_R := \sup_{|y|\leq R} g(y)$, then we get
\begin{align*}
    \Ee\left[ g^\epsilon(X_{t\wedge\sigma\wedge\tau_R})\wedge\kappa_R\right]
    &\leq g^\epsilon(0) + C_{\epsilon,g}^{b,Q,\nu}\Ee \left[\int_{[0,t\wedge\sigma\wedge\tau_R)} g(X_s)\wedge\kappa_R\,ds \right]\\
    &\leq g^\epsilon(0) + \tilde C_{\epsilon,g}^{b,Q,\nu} \int_{[0,t)} \Ee \left[g^\epsilon(X_{s\wedge\sigma\wedge\tau_R})\wedge\kappa_R\right]\,ds.
\end{align*}
We may now appeal to Gronwall's lemma and find
\begin{align*}
    \Ee\left[ g^\epsilon(X_{t\wedge\sigma\wedge\tau_R})\wedge\kappa_R\right]
    \leq g^\epsilon(0) \eup^{t C_{\epsilon,g}^{b,Q,\nu}}.
\end{align*}
Fatou's lemma proves  $\Ee\left[g^\epsilon(X_{t\wedge\sigma})\right]<\infty$, and \ref{doo-20-b} follows, if we take $\sigma\equiv t$.
\end{proof}

\section{Moments of stochastically continuous additive processes}\label{sec-add}

In this section, we apply Theorem \ref{gen-06} to characterize the existence of generalized moments of stochastically continuous additive processes $(X_t)_{t \geq 0}$.

\begin{theorem}\label{add-10}
    Let $(X_t)_{t \geq 0}$ be a stochastically continuous additive process, and denote by $(b_t,Q_t,\nu_t)$ its characteristics. Let $g \geq 1$ be a locally bounded submultiplicative  function. The following assertions are equivalent for any $t>0$.
	\begin{enumerate}\itemsep=6pt
		\item\label{add-10-a)} $\int_{|y \geq 1} g(y) \, \nu_t(dy)<\infty$,
		\item\label{add-10-b)} $\Ee g(X_t)<\infty$,
		\item\label{add-10-c)} $\sup_{s \leq t} \Ee g(X_s)<\infty$,
		\item\label{add-10-d)} $\Ee \left( \sup_{s \leq t} g(X_s) \right)<\infty$,
		\item\label{add-10-e)} $(g(X_{\sigma}))_{ \sigma\in\Tcal, \sigma\leq t}$ is uniformly integrable.
	\end{enumerate}
    If one, hence all, of these  conditions is satisfied, then there are positive constants $ C,C_1,C_2(t)$ such that
	\begin{equation}\label{eq-exp-3}
	   \Ee \left( \sup_{s \leq t} g(X_{s}) \right) \leq C \Ee g(X_{t})
	   \quad\text{and}\quad
	   \Ee \left( \sup_{s \leq t} g(X_{s}) \right) \leq C_1 \eup^{C_2(t)}.
	\end{equation}
	The constants $C,C_1,C_2(t)$ depend on $g$ and \textup{(}the characteristics of\textup{)} the additive process.
\end{theorem}

\begin{remark}\label{add-14}
\quad a)
    Klass \& Yang \cite{klass} prove the inequality $$\Ee \left( \sup\limits_{s \leq t} g(X_s) \right) \leq C \sup\limits_{s \leq t} \Ee g(X_t)$$ for so-called moderate functions $g$, i.e.\ functions satisfying for every $x$ and $y$ in $\real^d$ $g(x+y)\leq C_g(g(x)+g(y))$.

\medskip b)
   	The exponential moment $\Ee \exp(X_t)$ of an additive process can be calculated explicitly; namely, $\Ee \exp(X_t)= \exp(\psi_t)$ for the continuous negative definite function $\psi_t$ with triplet $(b_t,Q_t,\nu_t)$. If $(X_t)_{t \geq 0}$ is stochastically continuous, this follows from the fact that there is for each $t>0$ a L\'evy process $(Y_s)_{s \geq 0}$ with $Y_t = X_t$ in distribution; thus, $\Ee \exp(X_t)=\Ee \exp(Y_t)$ can be calculated using the well known formula for L\'evy processes. This observation simplifies part of the proof in Fujiwara \cite{fujiwara}, who also considers general additive processes which are neither stochastically continuous nor semimartingales.
\end{remark}

\begin{proof}[Proof of Theorem~\ref{add-10}]
We show
\ref{add-10-b)}$\Rightarrow$%
\ref{add-10-a)}$\Rightarrow$%
\ref{add-10-c)}$\Rightarrow$%
\ref{add-10-d)}$\Rightarrow$%
\ref{add-10-e)}$\Rightarrow$\ref{add-10-b)}.

\medskip\textbf{\primo}
 \ref{add-10-b)}$\Rightarrow$\ref{add-10-a)}. Since the law of $X_t$ is infinitely divisible, there exists a L\'evy process $(Y_s)_{s \geq 0}$ such that $Y_t=X_t$ in distribution. Therefore, the assertion follows directly from Theorem \ref{gen-06}, \ref{gen-06-a}$\Rightarrow$\ref{gen-06-e}.

\medskip\textbf{\secundo}
\ref{add-10-a)}$\Rightarrow$\ref{add-10-c)}.
The alternative proof of Corollary~\ref{doo-20}.\ref{doo-20-c}$\Rightarrow\ref{doo-20-b}$ in combination with Lemma \ref{intro-05} shows that
\begin{gather*}
    \Ee\left[g(X_s)\right] \leq c_\epsilon g^\epsilon(0) \eup^{s C_\epsilon C(b_s,Q_s,\nu_s;g)}
\end{gather*}
for the constant $C(b_s,Q_s,\nu_s;g)$ appearing in Lemma~\ref{gen-09}, where $g^{\epsilon}$ is the Friedrichs regularization of $g$ for some fixed $\epsilon>0$. Since $(b_s,Q_s,\nu_s)_{s\geq 0}$  are the characteristics of an additive process, the constant $C(b_s,Q_s,\nu_s; 1)$ (from Lemma~\ref{gen-09}) depends continuously on $s\geq 0$, see \eqref{intro-e05}; therefore it is bounded on $[0,t]$. The condition \ref{add-10-a)} ensures that $\int_{|y|\geq 1} g(y)\,\nu_s(dy) \leq \int_{|y|\geq 1} g(y)\,\nu_t(dy)<\infty$. Together we see that $\sup_{s\leq t} C(b_s,Q_s,\nu_s;g) \leq C(t) < \infty$, and
\begin{gather}\label{add-e11}
    \sup_{s\leq t}\Ee\left[g(X_s)\right] \leq C_1 \eup^{t C_{\epsilon} C(t)}.
\end{gather}
This bound can also be obtained from the estimate in \cite[Theorem 4.1]{kuehn-spa}.

\medskip\textbf{\tertio}
\ref{add-10-c)}$\Rightarrow$\ref{add-10-d)}
The proof mimics the proof of Theorem \ref{gen-06}.\ref{gen-06-a}$\Rightarrow$\ref{gen-06-b}. As $(X_s)_{s\geq 0}$ is stochastically continuous, it is  uniformly stochastically continuous on compact time intervals, and there is some constant $b>0$ such that
\begin{align*}
	\inf_{s\in [0,t]}\Pp\left(g(X_s-X_t)<g(0)\eup^b\right)\geq \frac{1}{2}.
\end{align*}
Indeed, we have for all $s\in [0,t]$ and any $K>0$ that 
\begin{align*}
    &\Pp\left(g(X_s-X_t)\geq g(0)\eup^b\right)\\
    &\qquad\leq \Pp\left(g(X_s-X_t)\geq g(0)\eup^b,\: |X_t-X_s|\leq K\right) + \Pp\left(|X_t-X_s| > K\right).
\end{align*}
Since $(X_s)_{s\geq 0}$ is uniformly stochastic continuous on the compact interval $[0,t]$, we can pick $K>0$ so large that the second term on the right is smaller than $\frac 12$. Since $g$ is locally bounded, we can now choose $b=b(K)$ so large that the first probability becomes zero.
Choose a constant $c>0$ such that $g(x+y)\leq cg(x)g(y)$ for all $x,y\in \rd$. Set for a fixed $a>0$
\begin{align*}
	\sigma:=\inf\left\{s>0:g(X_s)>cg(0)\eup^{a+b}\right\}.
\end{align*}
Since $(X_s)_{s \geq 0}$ has independent increments, we argue as in the proof of Theorem \ref{gen-06}, see \eqref{gen-e10}, to find that
\begin{align*}
	\Pp\left(g(X_t)>\eup^{a}\right)
    &\geq \Pp\left(g(X_{\sigma}-X_t)<g(0)\eup^b,\:\sigma\leq t\right)\\
	&= \int_{\{\sigma\leq t\}} \Pp\left(g(X_s-X_t)<g(0)\eup^b\right)\big|_{s=\sigma}\,d\Pp\\
	&\geq \frac{1}{2}\Pp(\sigma\leq t).
\end{align*}
As $\left\{\sup_{s \leq t} g(X_s)>cg(0)\eup^b \eup^a\right\} \subseteq \{\sigma \leq t\}$, this implies
\begin{align*}
	\Pp\left(g(X_t)>\eup^a\right)\geq \frac{1}{2}\Pp\left(\sup_{s\leq t}g(X_s)>cg(0)\eup^b \eup^a\right).
\end{align*}
If $\gamma$ is such that $\eup^{\gamma} = cg(0)\eup^b$, then we see as in \eqref{gen-e15} that
\begin{align*}
	\Ee \left(\sup_{s\leq t} g(X_s)\right)
    \leq \eup^\gamma+2\eup^{\gamma}\Ee g(X_t)
    \leq 3\eup^\gamma \Ee g(X_t).
\end{align*}

\medskip\textbf{\quarto}
\ref{add-10-d)}$\Rightarrow$\ref{add-10-e)}$\Rightarrow$\ref{add-10-b)}. Trivial.

\medskip
The moment estimates \eqref{eq-exp-3} follow from \eqref{add-e11} and the estimate in Step \tertio.
\end{proof}

\begin{remark}
	If $(X_t)_{t \geq 0}$ is an additive process, then $\Ee g(X_t)<\infty$ does not imply, in general, $\Ee g(X_r)>\infty$ for $r>t$. Take e.g.\ an isotropic $\alpha$-stable L\'evy process $(L_t)_{t \geq 0}$ for $\alpha \in (0,2)$ and consider
	\begin{equation*}
		X_t := \begin{cases} 0, & t <1, \\ L_{t-1}, & t \geq 1. \end{cases}
	\end{equation*}
	Clearly, $(X_t)_{t \geq 0}$ is an additive process and $\Ee \eup^{X_t}<\infty$ for all $t \leq 1$ but for $t  > 1$ we have $\Ee \eup^{X_t}=\infty$.

\end{remark}

\section{Martingale methods for L\'evy processes}\label{sec-lat}

Let us give a further application of Theorem~\ref{gen-06} to the recurrence/transience behaviour of L\'evy processes. As a warm-up and in order to illustrate the method, we begin with a very short proof for the characterization of infinitely divisible random variables taking values in a lattice (Theorem~\ref{lat-33}). Interestingly, the proof for transience (Corollary~\ref{lat-39}) is essentially the same argument, the difference being, that Theorem~\ref{lat-33} uses the characteristic exponent $\psi$ of the process for $\xi\in\real$, while Corollary~\ref{lat-39} relies on (the extension of $\psi$ to certain) $\zeta\in \iup\real$. Both results are, of course, well-known (Sato~\cite[Section~24]{sato}, resp., Zhang \emph{et al.} \cite[Theorem 3.4]{zhang}) but the present proofs are substantially different and much shorter. Recall that a random variable $Y$ is infinitely divisible if, and only if, there is a L\'evy process $(X_t)_{t\geq 0}$ such that $Y\sim X_1$.

\begin{theorem}\label{lat-33}
    Let $(X_t)_{t\geq 0}$ be a one-dimensional L\'evy process with characteristic exponent $\psi$ and triplet $(b,Q,\nu)$, see \eqref{intro-e01}. The following assertions are equivalent \textup{(}for fixed $\beta\neq 0$\textup{)}
    \begin{enumerate}
    \item\label{lat-33-a}
        $\psi(\beta)=\iup\alpha$ for some $\alpha \in\real$.
    \item\label{lat-33-b}
        $\left|\Ee \left[\eup^{\iup\beta X_{t}}\right] \right| = 1$ for some, hence for all, $t>0$.
    \item\label{lat-33-c}
        $X_{t_0}+ \gamma$ has values in the lattice $2\pi\beta^{-1}\integer$  for some $t_0>0$ and some $\gamma\in\real$.
   \item\label{lat-33-d}
	   $X_t + \alpha\beta^{-1} t$ has values in the lattice $2\pi\beta^{-1}\integer$ for and $\alpha \in \real$ and all $t>0$.
    \item\label{lat-33-e}
        $\supp\nu \subset 2\pi\beta^{-1}\integer$, \quad $Q=0$ and 
        \begin{gather*}
            b=-\alpha\beta^{-1}+2\pi \beta^{-1}\!\!\sum\limits_{|k|<\beta/(2\pi)}\!\!k\nu(\{2\pi \beta^{-1}k\}).
        \end{gather*}
        
    \end{enumerate}
\end{theorem}

\begin{remark}\label{lat-35}
    Our proof shows that there is also a connection between $\gamma$ and $\alpha$ in Theorem \ref{lat-33}.\ref{lat-33-c} and \ref{lat-33-d}. If we use $t=t_0$ in \ref{lat-33-d}, it holds that the localization parameter $\gamma \in \alpha \beta^{-1} t_0 + 2\pi \beta^{-1}\integer$.
\end{remark}

\begin{proof}
    \ref{lat-33-a}$\Rightarrow$\ref{lat-33-b}. This follows from $\Ee \left[\eup^{\iup\beta X_t}\right] = \eup^{-t\psi(\beta)}$ for any $t>0$.

    \medskip
    \ref{lat-33-b}$\Rightarrow$\ref{lat-33-c}. We fix $t_0>0$ such that $\left|\Ee \left[\eup^{\iup\beta X_{t_0}}\right] \right| = 1$  and apply Lemma~\ref{lat-37} to $X= \beta X_{t_0}$. This shows that there exists some $\gamma \in \real$ such that $X$ is supported in $2\pi\beta^{-1}\integer -\gamma$.

    \medskip
    \ref{lat-33-c}$\Rightarrow$\ref{lat-33-d}. Let $t_0>0$ be as in \ref{lat-33-c} and $t>0$, $t\neq t_0$. We see that $\Ee\left[ \eup^{\iup \beta X_{t}}\right] = (\eup^{-t_0\psi(\beta)})^{t/t_0}$. As $|\eup^{-t_0\psi(\beta)}|=1$, we know that $\psi(\beta)\in \iup\real$. Lemma~\ref{lat-37} shows that $X_{t}$ is supported in $2\pi\beta^{-1}\integer + \iup t\psi(\beta)\beta^{-1}=2\pi\beta^{-1}\integer - t\Im\psi(\beta)\beta^{-1}$ for every $t>0$.

    \medskip
    \ref{lat-33-d}$\Rightarrow$\ref{lat-33-e}. Let $t>0$. A direct calculation shows that
    \begin{align*}
        \Ee\left[\eup^{\iup \beta X_t}\right]
        = \sum_{k\in \integer} \eup^{\iup 2\pi  k-\iup\alpha t}\Pp(X_t=2\pi\beta^{-1}k - \alpha\beta^{-1} t)
        =\eup^{-\iup\alpha t}.
    \end{align*}
    On the other hand, infinite divisibility entails
    \begin{gather*}
        \eup^{-\iup\alpha t}
        = \Ee \left[\eup^{\iup \beta X_t}\right]
        = \eup^{-t\psi(\beta)}\quad \text{for all\ \ } t>0.
    \end{gather*}
    This is only possible if $\psi(\beta)=\iup\alpha$. Comparing this with the L\'evy--Khintchine formula \eqref{intro-e01}, we can conclude that \begin{gather*}
        b=-\alpha\beta^{-1} + 2\pi \beta^{-1}\sum_{|k|<\beta (2\pi)^{-1}}k\nu(\{2\pi \beta^{-1}k\}),
    \end{gather*} 
    $Q=0$ and $\supp\nu\subset 2\pi \beta^{-1}\integer$.

    \medskip
    \ref{lat-33-e}$\Rightarrow$\ref{lat-33-a}. This follows from the L\'evy--Khintchine formula \eqref{intro-e01}.
\end{proof}

The key step in the proof of Theorem~\ref{lat-33} is the following well-known result. The standard proof can be found in Lukacs~\cite[Section~2.1]{lukacs}, we prefer to give a(n equally short) martingale argument which appears, again, in the proof of Corollary~\ref{lat-39} and highlights the connection between Theorem~\ref{lat-33} and Corollary~\ref{lat-39}.

\begin{lemma}\label{lat-37}
    Let $X$ be a real random variable. If there exists $\beta\in \real\setminus\{0\}$ such that $\Ee \left[\eup^{\iup \beta X}\right]= \eup^{\iup\theta}$ for some $\theta \in \real$, then the distribution of $X$ is supported on $2\pi \beta^{-1}\integer+\beta^{-1}\theta$.
\end{lemma}
\begin{proof}
    Without loss of generality we may assume that $\beta=1$, i.e.\ it holds that $\left|\Ee \left[\eup^{\iup  X}\right]\right|=1$ or $\Ee \left[\eup^{\iup  (X-\theta)}\right]=1$ for some $\theta\in [0,2\pi)$. Let $(X_i)_{i \in \nat}$ be iid copies of the random variable $X$ and define for every $n\in\nat$
    \begin{gather*}
        Y_n
        := \frac 1{2^n}\prod_{k=1}^n(1+\cos(X_k -\theta))
        =\prod_{k=1}^n \frac{1+\cos(X_k -\theta)}{2}.
    \end{gather*}
    Set $\Fscr_n:=\sigma(X_1,\dotso,X_n)$ and note that $(Y_n)_{n\in\nat}$ is adapted to this filtration. We see that
    \begin{align*}
        \Ee[ Y_n \mid \mathcal{F}_{n-1}]
        &= \Ee\left[ \prod_{k=1}^n \frac{1+\cos(X_k-\theta)}{2} \mid \mathcal{F}_{n-1}\right]\\
        &= Y_{n-1} \frac{1+\Ee\left[\cos(X_n-\theta)\right]}{2}
        = Y_{n-1},
    \end{align*}
    so $(Y_n)_{n\in \nat}$ is a discrete-time martingale. As $0\leq Y_n \leq 1$ and $\Ee Y_n = 1$, we conclude that $Y_n$ converges to $1$ in $L^1$ and a.s.\ as $n\to\infty$. Thus, $Y_1=\Ee[1|\mathcal{F}_1]=1$ a.s., which implies that $X -\theta$ takes only values in the lattice $2\pi\integer$.
\end{proof}
Assume now that $(X_t)_{t\geq 0}$ is a L\'evy process which admits an exponential moment $\Ee\left[\eup^{\beta X_t}\right]<\infty$ for some $\beta\neq 0$. By Theorem~\ref{gen-06}, $\int_{|y|\geq 1} \eup^{\beta y}\,\nu(dy)<\infty$, and it is easy to see from the L\'evy--Khintchine formula \eqref{intro-e01} that the exponent $\psi$ has a continuous continuation to all complex numbers $\xi+\iup \eta\in\comp$ with $\xi\in\real$ and $\eta$ between $0$ and $-\beta$. In particular,
\begin{gather*}
    \Ee\left[\eup^{\beta X_t}\right] = \eup^{-t\psi(-\iup\beta)},\quad t>0,
\end{gather*}
which shows that the sets $A := \left\{\beta\in\real \mid \Ee\left[\eup^{\beta X_t}\right] = 1\right\}$, where $t>0$ is fixed, $\{\beta\in\real \mid \psi(-\iup\beta)=0\}$ and $\left\{\beta\in\real \mid \forall t>0\::\: \Ee\left[\eup^{\beta X_t}\right] = 1\right\}$ coincide. Using that $(X_t)_{t \geq 0}$ has stationary and independent increments, it follows with a similar argument as in the proof of Lemma~\ref{lat-37} that $A=\{\beta \in \real \mid (\eup^{\beta X_t})_{t \geq 0}$ is a martingale$\}$.

\begin{corollary}\label{lat-39}
    Let $(X_t)_{t\geq 0}$ be a one-dimensional L\'evy process. If the set $A\setminus\{0\}$ is not empty, then $X_t$ is transient.
\end{corollary}
\begin{proof}
    Define $Y_t=\eup^{\beta X_t}$ with $\beta\in A\setminus\{0\}$. The discussion before Corollary~\ref{lat-39} shows that $Y_t$ is a martingale. Since $t\mapsto\eup^{\beta Y_t}$ is positive and right-con\-tinu\-ous, the martingale convergence theorem shows that $\lim_{t\to\infty}Y_t = Y_\infty$ a.s.\ for some a.s.\ finite random variable $Y_\infty$.  As $\beta X_t$ is again a L\'evy process, which is either transient or recurrent, we see that $\eup^{\beta X_t}$ can only converge to a finite limit if $\beta X_t\to -\infty$ as $t\to\infty$; thus,  $X_t$ cannot be recurrent.
\end{proof}

    It is clear that Corollary~\ref{lat-39} still holds for a $d$-dimensional L\'evy process if we interpret $\xi X_t$ and $\beta X_t$ as scalar products with $\xi,\beta\in\rd$. By Cauchy's inequality, $ |\beta\cdot X_t|/|\beta|\leq |X_t|$, and so $\lim_{t\to\infty}|X_t|=\infty$, i.e.\ $(X_t)_{t\geq 0}$ is transient if $\beta\cdot X_t$ is transient.

\end{document}